\documentclass[final,leqno]{siamltex}
 \usepackage{epsfig,subeqn,bbold}
 \title{A New Error in Variables Model for Solving Positive Definite Linear System Using Orthogonal Matrix Decompositions}

\author{Negin Bagherpour \thanks{Faculty of Mathematical Sciences, Sharif University of Technology, Tehran, Iran, (nbagherpour@mehr.sharif.ir).}
\and Nezam Mahdavi-Amiri \thanks{Faculty of Mathematical Sciences, Sharif University of Technology, Tehran, Iran, (nezamm@sharif.ir).}}

\newcommand{\tr}{\mathop{\mathrm{tr}}}
\newcommand{\Null}{\mathop{\mathrm{Null}}}

\begin{document}
\maketitle
\begin{abstract}
The need to estimate a positive definite solution to an overdetermined linear system of equations with multiple right hand side vectors arises in several process control contexts. The coefficient and the right hand side matrices are respectively named data and target matrices. A number of optimization methods were proposed for solving such problems, in which the data matrix is unrealistically assumed to be error free. Here, considering error in measured data and target matrices, we present an approach to solve a positive definite constrained linear system of equations based on the use of a newly defined error function. To minimize the defined error function, we derive necessary and sufficient optimality conditions and outline a direct algorithm to compute the solution. We provide a comparison of our proposed approach and two existing methods, the interior point method and a method based on quadratic programming. Two important characteristics of our proposed method as compared to the existing methods are computing the solution directly and considering error both in data and target matrices. Moreover, numerical test results show that the new approach leads to smaller standard deviations of error entries and smaller effective rank as desired by control problems. Furthermore, in a comparative study, using the Dolan-Mor\'{e} performance profiles, we show the approach to be more efficient.\\
\end{abstract}

\markboth{NEGIN BAGHERPOUR AND NEZAM MAHDAVI-AMIRI}{Solving Positive Definite Total Least Squares Problems}

\section{Introduction}
Computing a symmetric positive definite solution of an overdetermined linear system of equations arises in a number of physical problems such as estimating the mass inertia matrix in the design of controllers for solid structures and robots; see, e.g., \cite{9}, \cite{17}, \cite{14}. Modeling a deformable structure also leads to such a mathematical problem; e.g., see \cite{25}. The problem turns into finding an optimal solution of the system
\begin{equation}\label{1}
  DX \simeq T,
\end{equation}
where $D,T \in {\mathbb{R}}^{m \times n}$, with $m\geq n$, are given and a symmetric positive definite matrix
$X \in {\mathbb{R}}^{n \times n}$ is to be computed as a solution of (\ref{1}). In some special applications, the data matrix $D$ has a simple structure, which may be taken into consideration for efficiently organized computations. Estimation of the covariance matrix and computation of the correlation matrix in finance are two such examples where the data matrices are respectively block diagonal and the identity matrix; e.g., see \cite{31}.\\
A number of least squares formulations have been proposed for physical problems, which may be classified as ordinary and error in variables (EIV) models.\\
Also, single or multiple right hand side least squares may arise. With a single right hand side, we have an overdetermined linear system of equations $Dx\simeq t$, where $D \in {\mathbb{R}}^{m \times n}$, $t \in {\mathbb{R}}^{m \times 1}$, with $m\geq n$, are known and the vector $x\in {\mathbb{R}}^{n \times 1}$ is to be computed. In an ordinary least squares formulation, the error is only attributed to $t$. So, to minimize the corresponding error, the following mathematical problem is devised:
\begin{eqnarray}\nonumber
&\min& \|\Delta t\| \\
&s. t.&  Dx=t+\Delta t. \label{2}
\end{eqnarray}
There are a number of methods for solving (\ref{2}), identified as direct and iterative methods. A well known direct method is based on using the QR factorization of the matrix $D$ \cite{27}. An iterative method has also been introduced in \cite{7} for solving (\ref{2}) using the GMRES algorithm.\\
In an EIV model, however, errors in both $D$ and $t$ are considered; e.g., see \cite{3}. Total least squares formulation is a well-known EIV model, where the goal is to solve the following mathematical problem (e.g., see \cite{6} and \cite{16}):
\begin{eqnarray}\nonumber
&\min& \|[\Delta D, \Delta t]\| \\
&s. t.&  (D+\Delta D)x=t+\Delta t. \label{3}
\end{eqnarray}

We note that {$\|\cdot\|$} in (\ref{2}) and (\ref{3}) respectively denote the vector 2-norm and the matrix Frobenius norm. Both direct \cite{24} and iterative \cite{12} methods have been presented for solving (\ref{3}). Moreover, the scaled total least squares formulation has been considered to unify both ordinary and total leats squares formulation; e.g., see \cite{24}. In a scaled toal least squares formulation, the mathematical problem 
\begin{eqnarray}\label{3333}
& \min \|[\Delta D, \Delta t]\| \nonumber\\
& s. t. (D+\Delta D)x=\lambda t+\Delta t
\end{eqnarray}
is to be solved for an arbitrary scalar $\lambda$. Zhou \cite{19} has studied the effect of perturbation and gave an error analysis of such a formulation.\\

A least squares problem with multiple right hand side vectors can also be formulated as an overdetermined system of equations $DX\simeq T$, where $D \in {\mathbb{R}}^{m \times n}$, $T \in {\mathbb{R}}^{m \times k}$, with $m\geq n$, are given and the matrix $X \in {\mathbb{R}}^{n \times k}$ is to be computed. With ordinary and total least squares formulations, the respective mathematical problems are:
\begin{eqnarray}\nonumber
&\min& \|\Delta T\| \\
&s. t.&  DX=T+\Delta T \nonumber\\
&& X \in {\mathbb{R}}^{n \times k}\label{4}
\end{eqnarray}
and
\begin{eqnarray}\nonumber
&\min& \|[\Delta D, \Delta T]\| \\
&s. t.&  (D+\Delta D)X=T+\Delta T \nonumber\\
&& X \in {\mathbb{R}}^{n \times k}.\label{5}
\end{eqnarray}

Common methods for solving (\ref{4})  are similar to the ones for (\ref{2}); see, e.g., \cite{7}, \cite{27}. Solving (\ref{5}) is possible by using the method described in \cite{8}, based on the SVD factorization of the matrix $[D,\hspace{0.1cm} T]$. Connections between ordinary least squares and total least squares formulations have been discussed in \cite{11}.\\

Here, we consider a newly defined EIV model for solving a positive definite linear problem. Our goal is to compute a symmetric positive definite solution $X \in {\mathbb{R}}^{n \times n}$ to the overdetermined system of equations $DX\simeq T$, where both matrices $D$ and $T$ may contain errors. We refer to this problem as positive definite linear system of equations later. No EIV model, even the well-known total least squares formulation, is considered for solving the positive definite linear system of equations in the literature. Several approaches have been proposed for this problem, commonly considering the ordinary least squares formulation and minimizing the error ${\| \Delta T\|}_F$ over all $n \times n$ symmetric positive definite matrices, where ${\|.\|}_F$ is the Frobenious norm; see e.g. \cite{10,23}.
Larson \cite{13} discussed a method for solving a positive definite least squares problem considering the corresponding normal system of equations. He considered both symmetric and positive definite least squares problems. Krislock \cite{25} proposed an interior point method for solving a variety of least squares problems with positive semi-definite constraints. Woodgate \cite{18} described a new algorithm for solving a similar problem in which a symmetric positive semi-definite matrix $P$ is computed to minimize $\|F-PG\|$, with known $F$ and $G$. Hu \cite{10} presented a quadratic programming approach to handle the positive definite constraint. In her method, the upper and lower bounds for the entries of the target matrix can be given as extra constraints.
In real measurements, however, both the data and target matrices may contain errors; hence, the total least squares formulation appears to be appropriate.\\
The rest of our work is organized as follows. In Section 2, we define a new error function and discuss some of its characteristics. A method for solving the resulting optimization problem with the assumption that $D$ has full
column rank is presented in Section 3. In Section 4, we generalize the method to the case of data matrix having an arbitrary rank.
In Section 5, a detailed discussion is made on computational complexity of both methods. Computational results and comparisons with available methods are given in Section 6. Section 7 gives our concluding remarks.

\section{Problem Formulation}
Consider a single equation $ax\simeq b$, where $a,b \in
{\mathbb{R}}^{n}$ and $x\in {\mathbb{R}}^+$. Errors in the $i$th entry of $b$ and $a$ are respectively equal to
$\mid a_ix-b_i \mid$ and $\mid a_i-\frac{b_i}{x} \mid$; e.g., see \cite{24}.\\
In \cite{24}, ${\sum}_{i=1}^{n} L_i$ was considered as a value to represent errors in both $a$ and $b$. As shown in Figure \ref{f1}, $L_i$ is the height of the triangle ABC which turns to be equal to $L_i=\frac{|b_i-a_ix|}{\sqrt{1+x^2}}$. Here, to represent the errors in both $a$ and $b$, we define the area error to be
\begin{equation}\label{77}
{\sum}_{i=1}^{n} |b_i-a_i x||a_i-\frac{b_i}{x}|,
\end{equation}
which is equal to
\[{\sum}_{i=1}^{n} (b_i-a_i x)(a_i-\frac{b_i}{x}),\]
for $x\in {\mathbb{R}}^+$.\\

Considering the problem of finding a symmetric and positive definite solution to the overdetermined system of linear equations $DX \simeq T$, in
which both $D$ and $T$ include error, the values $DX$ and $TX^{-1}$ are predicted values for $T$ and $D$ from the model $DX\simeq T$; hence, vectors ${\Delta T}_{j}={(DX-T)}_j$ and ${\Delta D}_{j}={(D-TX^{-1})}_{j}$ are the entries of errors in the $j$th column of $T$ and $D$, respectively.
Extending the error formulation (\ref{77}), the value
\begin{equation}\nonumber
E={\sum}_{j=1}^{n}(DX_j-T_j)^T(D_j-(TX^{-1})_j
\end{equation}
seems to be an appropriate measure of error. We also have
\begin{equation}\label{88}
E={\sum}_{j=1}^{n}{\sum}_{i=1}^{m}{(DX-T)}_{ij}{(D-TX^{-1})}_{ij}= \tr ((DX-T)^T (D-TX^{-1})),
\end{equation}
with $\tr(.)$ standing for trace of a matrix. Therefore, the problem can be formulated as
 \begin{equation}\label{6}
  \min\limits_{X\succ 0} \tr((DX-T)^T (D-TX^{-1})),
\end{equation}
where $X$ is symmetric and by $X\succ 0$, we mean $X$ is positive
definite. Problem (9) poses a newly defined EIV model for solving the positive definite linear system of equations.\\
In Lemma \ref{9}, we represent an equivalent formulation for the error, $E$. First, consider to a well-known property of positive definite matrices.
\textbf{Note}
A matrix $X\in {\mathbb{R}}^{n\times n}$ is positive definite if and only if there exists a nonsingular matrix $Y\in {\mathbb{R}}^{n\times n}$ such that $X=YY^T$.\\\\

The following results about the trace operator are also well-known; e.g., see \cite{21}.\\

\begin{lemma}\label{8}
For an nonsingulartible matrix $P \in {\mathbb{R}}^{n \times n}$ and arbitrary matrices $Y \in {\mathbb{R}}^{n \times
n}$, $A\in {\mathbb{R}}^{m \times n}$ and $B \in {\mathbb{R}}^{n \times m}$ we have\\\\
$\indent (1) \indent \tr(Y)=\tr(P^{-1}YP).$\\
$\indent (2) \indent \tr(AB)=\tr(BA).$\\\\
\end{lemma}

\begin{lemma}\label{9}
The error $E$, defined by (\ref{88}), is equal to
\begin{equation}\label{10}
E=\|DY-TY^{-T}\|_F^2
\end{equation}
where $X=YY^T$ and $\|.\|_F$ denotes the Frobenius norm of a matrix.
\end{lemma}
\begin{proof}
Substituting $X=YY^T$ in (\ref{88}) and using Lemma \ref{8}, we get
\begin{eqnarray}\nonumber
E&=&\tr((DX-T)^T (D-TX^{-1}))=\tr((DX-T)^T (DX-T)X^{-1})\nonumber\\
&=&\tr((DX-T)^T (DX-T)Y^{-T}Y^{-1})=\tr (Y^{-1}(DX-T)^T (DX-T)Y^{-T})\nonumber\\
&=&\tr ({(DXY^{-T}-TY^{-T})}^T (DXY^{-T}-TY^{-T}))\nonumber\\
&=&\tr( {(DY-TY^{-T})}^T(DY-TY^{-T}))\nonumber\\
&=&\|DY-TY^{-T}\|_F^2.\nonumber
\end{eqnarray}
\end{proof}
Considering this new formulation for $E$, it can be concluded that by use of our newly defined EIV model, computing a symmetric and positive definite solution to the over-determined system of equations $DX \simeq T$ is equivalent to computing a nonsingular matrix $Y\in {\mathbb{R}}^{n\times n}$ to be the solution of
\[\min \|DY-TY^{-T}\|_F^2,\]
and letting $X=YY^T$. A similar result is obtained by considering the over-determined system $DX \simeq T$ with $X=YY^T$ and multiplying both sides by $Y^{-T}$. We have,
\[DYY^T\simeq T,\]
or equivalently,
\begin{equation}\label{11}
DY \simeq TY^{-T}.
\end{equation}
Now, to assign a solution to (\ref{11}), it makes sense to minimize the norm of residual. Thus, to compute $X=YY^T$, it is sufficient to let $Y$ to be the solution of 
\[\min \|DY-TY^{-T}\|_F^2.\]

\textbf{Note} An appropriate characteristic of the error formulation proposed by (\ref{88}) is that for a symmetric and positive definite matrix $X$, the value of $E$ is nonnegative and it is equal to zero if and only if $DX=T$.\\\\
\section{\textbf{Mathematical Solution: Full Rank Data Matrix}}
Here, we are to develop an algorithm for solving
(\ref{6}) with the assumption that $D$ has full column
rank.\\
Using Lemma ~\ref{8}, with $X$ being symmetric, we have
$$\tr({(DX-T)}^T(D-TX^{-1}))=\tr(D^TDX+X^{-1}T^TT)-2\tr(T^TD).$$\\
So, (\ref{6}) can be written as
\begin{equation}\label{12}
\min \tr(AX+X^{-1}B),
\end{equation}
where $A=D^TD$ and $B=T^TT$ and the symmetric and positive definite matrix $X$ is to be computed.\\

\begin{corollary}
For each $X^{\ast}$ satisfying the first order necessary conditions of (\ref{12}), the sufficient optimality conditions described in Theorem ~\ref{14} are satisfied and since $\Phi(X)=\tr(AX+X^{-1}B)$ is convex on the cone of symmetric positive definite matrices, we can confirm that the symmetric positive definite matrix satisfying the KKT necessary conditions mentioned in Theorem ~\ref{13} is the unique global solution of (\ref{12}).\\
\end{corollary}
\subsection*{\textbf{Computing the positive definite matrix satisfying KKT conditions}}
As mentioned in Theorem \ref{13}, the KKT conditions lead to the nonlinear matrix equation 
\begin{equation}\label{116}
XAX=B.
\end{equation}
Note that ({\ref{116}}) is an special case of the continuous time Riccati equation (CARE), \cite{22}
\begin{equation}\label{115}
A^TXE+E^TXA-(E^TXB+S)R^{-1}(B^TXE+S^T)+Q=0,
\end{equation}
with $R=0$, $E=\frac{A}{2}$ and $Q=-B$. There is a MATLAB routine to solve CARE for arbitrary values of $A$, $E$, $B$, $S$, $R$ and $Q$. To use the routine, it is sufficient to type the command

\begin{verbatim}
X=care(A,B,Q,R,S,E),
\end{verbatim}
for the input arguments as in (\ref{115}). Higham \cite{22} developed an effective method for computing the positive definite solution to this special CARE when $A$ and $B$ are symmetric and positive definite using well-known decompositions. Lancaster and Rodman (\cite{28}) also discussed solving different types of algebraic Riccati equations. Moreover, they derived a perturbation analysis for these matrix equations. \\

\textbf{Note}
(QR decomposition) The QR decomposition \cite{27} of a matrix $A \in
{\mathbb{R}}^{m\times n}$ with $m\geq n$, is a decomposition of the form $A =QR$, where $R$ is an $m\times n$ upper triangular matrix and $Q$ satisfies $QQ^T=Q^TQ=I$. Moreover, if $A$ has full column rank, then $R$ also has full column rank.\\\\

\textbf{Note}
(Cholesky decomposition) A Cholesky decomposition \cite{27} of a symmetric positive definite matrix $A \in
{\mathbb{R}}^{n\times n}$ is a decomposition of the form $A =
R^TR$, where $R$, known as the Cholesky factor of $A$, is an $n\times n$ nonsingular upper triangular matrix.\\\\

\textbf{Note}
(Spectral decomposition) \cite{27}
All eigenvalues of a symmetric matrix, $A\in {\mathbb{R}}^{n\times n}$, are real and there exists an orthonormal matrix with columns representing the corresponding eigenvectors. Thus, there exist an orthonormal matrix $U$ with columns equal to the eigenvectors of $A$ and a diagonal matrix $D$ containing the eigenvalues such that $A=UDU^T$. Also, if $A$ is positive definite, then all of its eigenvalues are positive, and so we can set $D=S^2$. Thus, spectral decomposition for a symmetric positive definite matrix $A$ is a decomposition of the form $A=US^2U^T$, where $U^TU=UU^T=I$ and $S$ is a diagonal matrix.\\\\
\begin{theorem}\cite{22}\label{70}
Assume $D, T \in {\mathbb{R}}^{m\times n}$ with $m\geq n$ are known and $rank(D)=rank(T)=n$. Let $D=QR$ be the QR factorization of $D$. Let $A=D^TD$ and $B=T^TT$. Define the matrix $\tilde{Q}=RBR^T$ and compute its spectral decomposition, that is,
$\tilde{Q}=RBR^T=U{\tilde{S}}^2U^T$. Then, (\ref{12}) has a unique solution, given by
$$X^{\ast}=R^{-1}U\tilde{S}U^TR^{-T}.$$
\end{theorem}
\begin{proof}
Based on Theorem \ref{14} and the afterwards discussion, it is sufficient to show that $X^{\ast}$ satisfies the necessary optimality conditions, $X^{\ast}AX^{\ast}=B$. Note that from $D=QR$, we have \[A=D^TD=R^TQ^TQR=R^TR.\]
Substituting $X^{\ast}$, we have\\
\begin{eqnarray}\nonumber
X^{\ast}AX^{\ast}&=&R^{-1}U\tilde{S}U^TR^{-T}R^TRR^{-1}U\tilde{S}U^TR^{-T}\nonumber\\
&=&R^{-1}U{\tilde{S}}^2U^TR^{-T}=R^{-1}RBR^TR^{-T}=B.\nonumber
\end{eqnarray}
\end{proof}
\textbf{Note}
To compute $R$, it is also possible to first compute $A=D^TD$ and then calculate the Cholesky decomposition for $A$. However, because of more stability, in Theorem \ref{70} the QR decomposition of $D$ is used.\\\\
We are now ready to outline the steps of our proposed algorithm.\\\\

\textbf{Solving the EIV model for positive definite linear system using QR decomposition.}\\

Compute the QR decomposition for D and let $D=QR$.\\
Let $\tilde{Q}=RBR^T$, where $B=T^TT$ and compute the spectral decomposition of $\tilde{Q}$, that is, $\tilde{Q}=U{\tilde{S}}^2U^T.$\\
 Set $X^{\ast}=R^{-1}U\tilde{S}U^TR^{-T}.$\\
 Set $E=\tr((DX^{\ast}-T)^T(D-T{X^{\ast}}^{-1}))$.\\\\

Note that \textsc{Algorithm 1} computes the solution of (\ref{6})
directly.\\

The following theorem shows that by use of spectral decomposition of $A$ a method similar to the one introduced in \cite{22} is in hand for solving the continuous time Riccati equation.\\
\begin{theorem}\label{18}
Let $A=D^TD$ and $B=T^TT$ with $D,T \in
{\mathbb{R}}^{m \times n}$, $m\geq n$ and
$rank(D)=n.$ Let the spectral decomposition of $A$ be
$A=US^2U^T$. Define the matrix $\tilde{Q}=SU^TBUS$
 and compute its spectral decomposition, $\tilde{Q}=SU^TBUS=\bar{U}{\bar{S}}^2{\bar{U}}^T$.
Then, the unique minimizer of (\ref{12}) is $$X^{\ast}=US^{-1}\bar{U}\bar{S}{\bar{U}}^TS^{-1}U^T.$$
\end{theorem}
\begin{proof}
Similar to the proof of Theorem \ref{70}, it is sufficient to show that the mentioned $X^{\ast}$ satisfies $X^{\ast}AX^{\ast}=B$. Substituting $X^{\ast}$, we have

\begin{eqnarray}\nonumber
X^{\ast}AX^{\ast}&= US^{-1}\bar{U}\bar{S}{\bar{U}}^TS^{-1}U^TUS^2U^TUS^{-1}\bar{U}\bar{S}{\bar{U}}^TS^{-1}U^T\nonumber\\
&= US^{-1}\bar{U}{\bar{S}}^2{\bar{U}}^TS^{-1}U^T=US^{-1}SU^TBUSS^{-1}U^T=B.\nonumber
\end{eqnarray}
\end{proof}
Next, based on Theorem ~\ref{18}, we outline an algorithm for solving (\ref{6}).\\\\

\textbf{Solving the EIV model for positive definite linear system using spectral decomposition.}\\

 Let $A=D^TD$ and compute its spectral decomposition: $A=US^2U^T.$\\
 Let $\tilde{Q}=SU^TBUS$, where $B=T^TT$ and compute the spectral decomposition of $\tilde{Q}$, that is, $\tilde{Q}=\tilde{U}{\tilde{S}}^2{\tilde{U}}^T.$\\
 Set $X^{\ast}=US^{-1}\tilde{U}\tilde{S}{\tilde{U}}^TS^{-1}U^T.$\\
 Set $E=\tr((DX^{\ast}-T)^T(D-T{X^{\ast}}^{-1}))$.\\\\

In Section 4 we generalize our proposed method for solving positive definite linear system of equations when the data matrix is rank deficient.
\section{Mathematical Solution: Rank Deficient Data Matrix}
Since the data matrix $D$ is usually produced from experimental
measurements, we may have $rank(D)<n$. Here, we are to generalize \textsc{Algorithm 1} for solving (\ref{6}), assuming that $rank(D)=r<n$. In Section 4.1 we outline two algorithms to compute the general solution of (\ref{6}). It will be shown that, in general, (\ref{6}) may not have a unique solution. Hence, in section 4.2 we discuss how to find a particular solution of (\ref{6}) having desirable characteristics for control problems.\\
\subsection{General solution}
Based on theorems ~\ref{13} and ~\ref{14}, a symmetric positive definite matrix $X^{\ast}$ is a solution of (\ref{6}) if and only if
\begin{equation}\label{19}
X^{\ast}AX^{\ast}=B.
\end{equation}
Therefore, in the following, we discuss how to find a symmetric positive definite matrix $X^{\ast}$ satisfying (\ref{19}).\\
First we note that in case $D$ and $T$ are rank deficient, there might be no solution for (\ref{19}), and if there is any, it is not necessarily a unique solution; see, e.g., \cite{22}. Higham \cite{22} considered to $X=B^{\frac{1}{2}}{(B^{\frac{1}{2}}AB^{\frac{1}{2}})}^{-\frac{1}{2}}B^{\frac{1}{2}}$ as a solution of (\ref{19}), which is symmetric and positive semidefinite. However, we are interested in finding a symmetric positive definite solution to (\ref{19}). Hence, in the following, first the necessary and sufficient conditions on $A$ and $B$ to guarantee the existence of positive definite solution to (\ref{19}) are discussed. We then outline two algrithms to compute such a solution.

Let the spectral decomposition of $A$ be $A=U\left(
                                               \begin{array}{cc}
                                                 S^2 & 0 \\
                                                 0 & 0 \\
                                               \end{array}
                                             \right)
U^T$, where $S^2\in {\mathbb{R}}^{r \times r}$ is a diagonal matrix having the positive eigenvalues of $A$ as its diagonal entries. Substituting the decomposition in (\ref{19}), we get\\
\begin{equation}\label{20}
X^{\ast}U\left(
                                               \begin{array}{cc}
                                                 S^2 & 0 \\
                                                 0 & 0 \\
                                               \end{array}
                                             \right)U^TX^{\ast}=B.
\end{equation}
Since $U$ is orthonormal, (\ref{20}) can be written as
\begin{center}
$U^TX^{\ast}U\left(
                                               \begin{array}{cc}
                                                 S^2 & 0 \\
                                                 0 & 0 \\
                                               \end{array}
                                             \right)U^TX^{\ast}U=U^TBU.$
\end{center}
Then, letting $\tilde{X}=U^TXU$ and $\tilde{B}=U^TBU$, we have
\begin{equation}\label{21}
\tilde{X}\left(
                                               \begin{array}{cc}
                                                 S^2 & 0 \\
                                                 0 & 0 \\
                                               \end{array}
                                             \right)\tilde{X}=\tilde{B}.
\end{equation}
Thus, the matrix $X=U\tilde{X}U^T$ is a solution of (\ref{6}) if and only if $\tilde{X}$ is symmetric positive definite and satisfies (\ref{21}). Substituting the block form $\tilde{X}=\left(
                                         \begin{array}{cc}
                                           {\tilde{X}}_{rr} & {\tilde{X}}_{r,n-r} \\
                                           {\tilde{X}}_{n-r,r} & {\tilde{X}}_{n-r,n-r} \\
                                         \end{array}
                                       \right)
$, where ${\tilde{X}}_{rr}\in {\mathbb{R}}^{r \times r}$, ${\tilde{X}}_{r,n-r}={\tilde{X}}_{n-r,r}^T \in {\mathbb{R}}^{r \times (n-r)}$ and ${\tilde{X}}_{n-r,n-r}\in {\mathbb{R}}^{(n-r) \times (n-r)}$, in (\ref{21}) leads to
\begin{equation}\nonumber
\left(
                                               \begin{array}{cc}
                                                 {\tilde{X}}_{rr}S^2{\tilde{X}}_{rr} & {\tilde{X}}_{rr}S^2{\tilde{X}}_{r,n-r} \\
                                                 {\tilde{X}}_{n-r,r}S^2{\tilde{X}}_{rr} & {\tilde{X}}_{n-r,r}S^2{\tilde{X}}_{r,n-r} \\
                                               \end{array}
                                             \right)=\tilde{B}=\left(
                                         \begin{array}{cc}
                                           {\tilde{B}}_{rr} & {\tilde{B}}_{r,n-r} \\
                                           {\tilde{B}}_{n-r,r} & {\tilde{B}}_{n-r,n-r} \\
                                         \end{array}
                                       \right),
                                       \end{equation}
 which is satisfied if and only if\\
 \begin{subequations}
 \begin{equation}\label{22}
{\tilde{X}}_{rr}S^2{\tilde{X}}_{rr}={\tilde{B}}_{rr},
 \end{equation}
 \begin{equation}\label{23}
 {\tilde{X}}_{rr}S^2{\tilde{X}}_{r,n-r}={\tilde{B}}_{r,n-r},
 \end{equation}
  \begin{equation}\label{24}
{\tilde{X}}_{n-r,r}S^2{\tilde{X}}_{r,n-r}={\tilde{B}}_{n-r,n-r}.
 \end{equation}
 \end{subequations}
Before discussing how to compute $\tilde{X}$, we show that if (\mbox{\ref{19}}) has a symmetric and positive definite solution, then ${\tilde{B}}_{rr}$ must be nonsingular. The matrix ${\tilde{X}}_{rr}$ as a main minor of the positive definite matrix $\tilde{X}$ is nonsingular. $S$ is also supposed to be nonsingular. Hence, it can be concluded from ({\ref{22}}) that ${\tilde{B}}_{rr}$ is nonsingular.

Let $\bar{D}=S$ and suppose $\bar{T}$ satisfies ${\bar{T}}^T{\bar{T}}={\tilde{B}}_{rr}$. Consider problem (\ref{6}) corresponding to the data and target matrices $\bar{D}$ and $\bar{T}$ as follows:
\begin{equation}\label{25}
\min\limits_{\bar{X}\succ 0} \tr((\bar{D}\bar{X}-\bar{T})^T (\bar{D}-\bar{T}{\bar{X}}^{-1})).
\end{equation}
We know from theorems ~\ref{13} and ~\ref{14} that the necessary and sufficient optimality conditions for the unique solution of problem (\ref{25}) implies (\ref{22}). Thus, ${\tilde{X}}_{rr}$ can be computed using \textsc{Algorithm 1} for the input arguments $\bar{D}$ and $\bar{T}$. Substituting the computed ${\tilde{X}}_{rr}$ in (\ref{23}), the linear system of equations
\begin{equation}\label{26}
 {\tilde{X}}_{rr}S^2{\tilde{X}}_{r,n-r}={\tilde{B}}_{r,n-r}
 \end{equation}
 arises, where ${\tilde{X}}_{rr}, S^2 \in {\mathbb{R}}^{r \times r}$ are known and
 ${\tilde{X}}_{r,n-r} \in {\mathbb{R}}^{r \times (n-r)}$ is to be computed. Since ${\tilde{X}}_{rr}$ is positive definite and $S^2$ is nonsingular, the coefficient matrix of the linear system (\ref{26}) is nonsingular and ${\tilde{X}}_{r,n-r}$ can be uniquely computed.\\ It is clear that since $\tilde{X}$ is symmetric, ${\tilde{X}}_{n-r,r}$ is the same as ${{\tilde{X}}_{r,n-r}}^T$.
 Now, we check whether the computed ${\tilde{X}}_{n-r,r}$ and ${\tilde{X}}_{r,n-r}$ satisfy (\ref{24}). Inconsistency of (\ref{25}) means that there is no symmetric positive definite matrix satisfying (\ref{22})-(\ref{24}), and if so, (\ref{6}) has no solution. Thus, in solving an specific positive definite system with rank deficient data and target matrices using the presented EIV model, a straightforward method to investigate the existence of solution is to check whether (\ref{24}) holds for the given data and target matrices. On the other hand, for numerical results, it is necessary to generate meaningful test problems. Hence, in the following two lemmas, we investigate the necessary and sufficient conditions for satisfaction of (\ref{24}).\\
 \begin{lemma}\label{27}
 Let the spectral decomposition of $A$ be determined as 
 \[A=U\left(
                                               \begin{array}{cc}
                                                 S^2 & 0 \\
                                                 0 & 0 \\
                                               \end{array}
                                             \right)U^T\]
                                             where $S^2 \in {\mathbb{R}}^{r \times r}$ and $rank(A)=rank(B)=r$.
 The necessary and sufficient condition for satisfaction of (\ref{24}) is
\[BU_r{({U_r}^T BU_r)}^{-1}{U_r}^TB -B\in \Null({U_{n-r}}^T).\]
 \end{lemma}
 \begin{proof}

From (\ref{22}), we have
  \begin{equation}\label{29}
 {{\tilde{X}}_{rr}}^{-1}S^{-2}{{\tilde{X}}_{rr}}^{-1}={{\tilde{B}}_{rr}}^{-1},
 \end{equation}
and from (\ref{24}), we get
   \begin{equation}\label{30}
 {\tilde{X}}_{r,n-r}=S^{-2}{{\tilde{X}}_{rr}}^{-1}{{\tilde{B}}_{r,n-r}},
 \end{equation}
\begin{equation}\nonumber
 {\tilde{X}}_{n-r,r}={\tilde{B}}_{n-r,r}{{\tilde{X}}_{rr}}^{-1}S^{-2}.
 \end{equation}
Manipulating (\ref{24}) with (\ref{29}) and (\ref{30}), we get
\begin{equation}\label{31}
 {\tilde{B}}_{n-r,r}{{\tilde{B}}_{rr}}^{-1}{\tilde{B}}_{r,n-r}= {\tilde{B}}_{n-r,n-r}.
 \end{equation}
 Considering the block form $U=\left(
                                 \begin{array}{cc}
                                   U_r & U_{n-r} \\
                                 \end{array}
                               \right)$, where $U_r \in {\mathbb{R}}^{n\times r}$ and $U_{n-r} \in {\mathbb{R}}^{n\times (n-r)}$, we have
 \begin{eqnarray}
 \tilde{B}=U^TBU&=&\left(
                   \begin{array}{c}
                     {U_r}^T \\
                     {U_{n-r}}^T \\
                   \end{array}
                 \right)B
 \left(
                                 \begin{array}{cc}
                                   U_r & U_{n-r} \\
                                 \end{array}
                               \right)\nonumber \\
                               &=&\left(
                                         \begin{array}{cc}
                                           {U_r}^T BU_r & {U_r}^T BU_{n-r} \\
                                           {U_{n-r}}^T BU_r & {U_{n-r}}^T BU_{n-r} \\
                                         \end{array}
                                       \right).\nonumber
\end{eqnarray}
Rewriting (\ref{31}) results in
 \begin{equation}\label{332}
 {U_{n-r}}^T BU_r{({U_r}^T BU_r)}^{-1}{U_r}^T BU_{n-r}={U_{n-r}}^T BU_{n-r}, 
 \end{equation}
 which is equivalent to (e.g., see \cite{20})
 \begin{equation}\label{32}
 BU_r{({U_r}^T BU_r)}^{-1}{U_r}^TB =B+Z,
 \end{equation}
where $Z\in {\mathbb{R}}^{n \times n}$ is in the null space of ${U_{n-r}}^T$. Thus, (\ref{19}) has a positive definite solution if and only if 
\begin{equation}\label{3344}
BU_r{({U_r}^T BU_r)}^{-1}{U_r}^TB -B\in \Null({U_{n-r}}^T).
\end{equation}

\end{proof}
\textbf{Note}
For real problems with arbitrary values of $D$ and $T$, the necessary and sufficient condition given in Lemma \mbox{\ref{27}} may not be satisfied, in general. Hence, we are to propose a threshold to determine if
\begin{equation}
F={U_{n-r}}^T\left(BU_r{({U_r}^T BU_r)}^{-1}{U_r}^TB -B\right)
\end{equation}
is close enough to zero. In the following, we show that if $\|F\|<\delta$, for a sufficiently small scalar $\delta$, then $X_{r(n-r)}$ computed from (\ref{23}) is a proper approximation for the solution of (\ref{24}). Substituting $F$ in ({\ref{332}}), we have \begin{equation}
{\tilde{B}}_{n-r,r}{{\tilde{B}}_{rr}}^{-1}{\tilde{B}}_{r,n-r}- {\tilde{B}}_{n-r,n-r}=FU_{n-r},
\end{equation} 
and 
\begin{equation}
{\tilde{X}}_{n-r,r}S^2{\tilde{X}}_{r,n-r}-{\tilde{B}}_{n-r,n-r}=FU_{n-r}.
\end{equation}
Let $X^{\ast}$ satisfy (\ref{24}), that is,
\begin{equation}
{X^{\ast}}_{n-r,r}S^2{X^{\ast}}_{r,n-r}-{\tilde{B}}_{n-r,n-r}=0.
\end{equation}
Then, we have
\begin{equation}\label{331}
{\tilde{X}}_{n-r,r}S^2{\tilde{X}}_{r,n-r}-{X^{\ast}}_{n-r,r}S^2{X^{\ast}}_{r,n-r}=FU_{n-r}.
\end{equation}
Letting $\tilde{Y}=S{\tilde{X}}_{r,n-r}$ and $Y^{\ast}=S{X^{\ast}}_{r,n-r}$, (\ref{331}), we get
\begin{equation}
{\tilde{Y}}^T\tilde{Y}-{Y^{\ast}}^TY^{\ast}=FU_{n-r}
\end{equation}

and
\begin{equation}\label{337}
{\tilde{y}}_i^T{\tilde{y}}_j-{y_i^{\ast}}^Ty_j^{\ast}=(FU_{n-r})_{ij},
\end{equation}
where ${\tilde{y_i}}$ and ${y_i^{\ast}}$ are the $i$th column of $\tilde{Y}$ and $Y^{\ast}$ respectively. Now, since the 2 norm of each column of $U_{n-r}$ is equal to one, every entry of $U_{n-r}$ is less than or equal to one. Moreover, under the assumption $\|F\|<\delta$, none of the entries of $F$ are greater than $\delta$. Hence, we have
\begin{equation}\label{4444}
|{(FU_{n-r})}_{ij}|=|f_i^Tu_j|\leq |f_{i1}+\cdots +f_{i(n-r)}|< (n-r)\delta,
\end{equation}
where $f_i^T$ and $u_j$ are the $i$th row of $F$ and the $j$th column of $U_{n-r}$ respectively. Now, (\ref{337}) together with (\ref{4444}) gives
\begin{equation}
|{\tilde{y_i}}^T\tilde{y_j}-{y_i^{\ast}}^Ty_j^{\ast}|< {(n-r)}\delta.
\end{equation}
Hence, there is a constant $c_{ij}$ such that
\begin{equation}\label{555}
|{\tilde{y}}_{ij}-y_{ij}^{\ast}|< c_{ij},
\end{equation}
where $\tilde{y}_{ij}$ and $y_{ij}^{\ast}$ are the $(i,j)$th entry of $\tilde{Y}$ and $Y^{\ast}$ respectively. Letting $S=diag(s_1,\cdots,s_r)$, from (\ref{555}) we get
\begin{equation}
|s_i||({\tilde{X}}_{n-r,r})_{ij}-(X^{\ast}_{n-r,r})_{ij}|\leq c_{ij},
\end{equation}
for $i=1,\cdots,r$ and $j=1,\cdots,n-r$ and
\[{\|{\tilde{X}}_{r,n-r}-{X^{\ast}}_{r,n-r}\|}\leq C.\]
Hence, assuming
\[{\tilde{X}}_{r,r}={X^{\ast}}_{r,r},\]
we have $\|\tilde{X}-X^{\ast}\|<\alpha$ which means that if $FU_{n-r}$ is close enough to zero, then the computed solution from the approximate satisfaction of (\ref{24}) would be close enough to the exact solution.\\\\
In the following lemma, we give a sufficient condition which guarantees the existence of a solution for \mbox{(\ref{19})}. We later use this result to generate consistent test problems in Section 6.
\begin{lemma}\label{117}
Let the spectral decomposition of $B$ be $B=V\left(
                                               \begin{array}{cc}
                                                 {\sum}^2 & 0 \\
                                                 0 & 0 \\
                                               \end{array}
                                             \right)V^T,$ where ${\sum}^2 \in {\mathbb{R}}^{r \times r}$ and $rank(A)=rank(B)=r$.
A sufficient condition for satisfaction of (\ref{24}) is that
  \begin{equation}\label{28}
 V=U\left(
                                               \begin{array}{cc}
                                                 Q & 0 \\
                                                 0 & P \\
                                               \end{array}
                                             \right),
  \end{equation}
where $Q \in {\mathbb{R}}^{r \times r}$ and $P \in {\mathbb{R}}^{(n-r) \times (n-r)}$ satisfy $QQ^T=Q^TQ=I$ and $PP^T=P^TP=I$.
\end{lemma}

\begin{proof}
A possible choice for $Z$ in Lemma \ref{117} is zero, for which (\ref{32}) is equivalent to
  \begin{equation}\label{33}
 U_r{({U_r}^T BU_r)}^{-1}{U_r}^T=B^++W,
 \end{equation}
with $W\in {\mathbb{R}}^{n \times n}$ in the null space of $B$. To obtain a simplified sufficient condition for existence of a positive definite solution to (\ref{19}), we let $W=0$. Multiplying (\ref{33}) by ${U_r}^T$ and $U_r$ respectively on the left and right, and substituting the spectral decomposition of $B$, we get
   \begin{equation}\label{34}
{({U_r}^T V_r{\sum}^2{V_r}^TU_r)}^{-1}={U_r}^TB^+U_r={U_r}^TV_r{\sum}^{-2}{V_r}^TU_r.
  \end{equation}
Letting $M={U_r}^T V_r$, we get
    \begin{equation}\nonumber
{(M{\sum}^2M^T)}^{-1}=M{\sum}^{-2}M^T.
  \end{equation}
Since $M$ has full rank, we get
 \begin{equation}\label{35}
M^{-T}{\sum}^{-2}M^{-1}=M{\sum}^{-2}M^T. \\
\end{equation}
Now, since ${\sum}^{-2}$ is nonsingular, (\ref{35}) holds if and only if
\begin{equation}\label{36}
 M^TM=I.
\end{equation}
This leads to\\
    \begin{equation}\label{44}
{({U_r}^T V_r)}^T{U_r}^T V_r={V_r}^T U_r{U_r}^T V_r=I.
  \end{equation}
Since $U$ is orthonormal, we have $UU^T=U_r{U_r}^T+U_{n-r}{U_{n-r}}^T=I$. Hence, we get\\
\begin{equation}\label{45}
U_r{U_r}^T=I-U_{n-r}{U_{n-r}}^T.
\end{equation}
Substituting (\ref{45}) in (\ref{44}), we get
    \begin{equation}\nonumber
{V_r}^T (I-U_{n-r}{U_{n-r}}^T) V_r=I-{V_r}^TU_{n-r}{U_{n-r}}^T V_r=I,
  \end{equation}
which is satisfied if and only if ${U_{n-r}}^T V_r=0$. Since the columns of $U_r$ form an orthogonal basis for the null space of ${U_{n-r}}^T$ \cite{27}, it can be concluded that each column of $V_r$ is a linear combination of the columns of $U_r$. Thus,
   \begin{equation}\label{46}
V_r=U_rQ
  \end{equation}
is a necessary condition for (\ref{36}) to be satisfied, and since both $U_r$ and $V_r$ have orthogonal columns, $Q\in {\mathbb{R}}^{r \times r}$ satisfies $QQ^T=Q^TQ=I$. On the other hand, we know from the definition of the spectral decomposition that $VV^T=UU^T=I$. Thus,
 \begin{eqnarray}\nonumber
& V_r{V_r}^T+V_{n-r}{V_{n-r}}^T=I,\\
& U_r{U_r}^T+U_{n-r}{U_{n-r}}^T=I. \label{47}
\end{eqnarray}
Manipulating (\ref{46}) with (\ref{47}), we get
   \begin{equation}\label{48}
V_{n-r}{V_{n-r}}^T=U_{n-r}{U_{n-r}}^T,
  \end{equation}
which holds if and only if there exists a matrix $P\in  {\mathbb{R}}^{(n-r) \times (n-r)}$ such that $PP^T=P^TP=I$ and
\begin{equation}\label{49}
V_{n-r}=PU_{n-r}.
\end{equation}
It can be concluded from (\ref{46}) and (\ref{49}) that $V=U\left(
                                               \begin{array}{cc}
                                                 Q & 0 \\
                                                 0 & P \\
                                               \end{array}
                                             \right)$, where $QQ^T=Q^TQ=I$ and $PP^T=P^TP=I$.
\end{proof}

\begin{corollary}\label{89}
The matrices $P$ and $Q$ defined in Lemma \ref{27} can set to be rotation matrices \cite{27} to satisfy
 \begin{eqnarray}\nonumber
& PP^T=P^TP=I,\\
& QQ^T=Q^TQ=I. \nonumber
\end{eqnarray}
Thus, to compute a target matrix, $T$, satisfying Lemma \ref{27}, it is sufficient to first compute $V$ from (\ref{28}) with $Q\in {\mathbb{R}}^{r \times r}$ and $P \in {\mathbb{R}}^{(n-r) \times (n-r)}$ arbitrary rotation matrices and $U$ as defined in Lemma \ref{27} and then set $T=\bar{U}\left(
                                               \begin{array}{cc}
                                                 \sum & 0 \\
                                                 0 & 0 \\
                                               \end{array}
                                             \right)V^T$, where $\bar{U} \in {\mathbb{R}}^{m \times m}$ and $\sum \in {\mathbb{R}}^{r \times r}$ are arbitrary orthonormal and diagonal matrices.
                                             \end{corollary}
Thus, problem (\ref{6}) has a solution if and only if the data and target matrices satisfy (\ref{3344}).
In this case, ${\tilde{X}}_{rr}$, ${\tilde{X}}_{r,n-r}$ and its transpose, ${\tilde{X}}_{n-r,r}$, are respectively computed from (\ref{22}) and (\ref{23}). Hence, the only remaining step is to compute ${\tilde{X}}_{n-r,n-r}$ so that $\tilde{X}$ is symmetric and positive definite.\\
We know that $\tilde{X}$ is symmetric positive definite if and only if there exists a nonsingular lower triangular matrix $L\in {\mathbb{R}}^{n \times n}$ so that
\begin{equation}\label{37}
\tilde{X}=LL^T,
\end{equation}
where $L$ is lower triangular and nonsingular. Considering the block forms
$\tilde{X}=\left(
                                         \begin{array}{cc}
                                           {\tilde{X}}_{rr} & {\tilde{X}}_{r,n-r} \\
                                           {\tilde{X}}_{n-r,r} & {\tilde{X}}_{n-r,n-r} \\
                                         \end{array}
                                       \right)$ and $L=\left(
                                         \begin{array}{cc}
                                           L_{rr} & 0 \\
                                           L_{n-r,r} & L_{n-r,n-r} \\
                                         \end{array}
                                       \right)$, where $L_{n-r,r}$ is an $(n-r) \times r$ matrix and $L_{rr} \in {\mathbb{R}}^{r \times r}$ and $L_{n-r,n-r} \in {\mathbb{R}}^{(n-r) \times (n-r)}$ are nonsingular lower triangular matrices, we get
\begin{eqnarray}
\left(
                                         \begin{array}{cc}
                                           {\tilde{X}}_{rr} & {\tilde{X}}_{r,n-r} \\
                                           {\tilde{X}}_{n-r,r} & {\tilde{X}}_{n-r,n-r} \\
                                         \end{array}
                                       \right)&=& \nonumber \\
         &&   \left(
                                         \begin{array}{cc}
                                           L_{rr} & 0 \\
                                           L_{n-r,r} & L_{n-r,n-r} \\
                                         \end{array}
                                       \right)\left(
                                         \begin{array}{cc}
                                           {L_{rr}}^T & {L_{n-r,r}}^T \\
                                           0 & {L_{n-r,n-r}}^T \\
                                         \end{array}
                                       \right).\label{51}
\end{eqnarray}
 Thus,
\begin{subequations}
\begin{equation}\label{40}
  {\tilde{X}}_{rr}=L_{rr}{L_{rr}}^T,
\end{equation}
\begin{equation}\label{41}
{\tilde{X}}_{r,n-r}=L_{rr}{L_{n-r,r}}^T,
\end{equation}
\begin{equation}\label{42}
{\tilde{X}}_{n-r,r}=L_{n-r,r}{L_{rr}}^T,
\end{equation}
\begin{equation}\label{43}
{\tilde{X}}_{n-r,n-r}=L_{n-r,r}{L_{n-r,r}}^T+L_{n-r,n-r}{L_{n-r,n-r}}^T.
\end{equation}
\end{subequations}
Therefore, to compute a symmetric positive definite $\tilde{X}$, (\ref{40})--(\ref{43}) must be satisfied. Let ${\tilde{X}}_{rr}=\tilde{L}{\tilde{L}}^T$ be the Cholesky decomposition of ${\tilde{X}}_{rr}$. $L_{rr}=\tilde{L}$ satisfies (\ref{40}). Substituting $L_{rr}$ in (\ref{41}), ${L_{n-r,r}}^T$ is computed uniquely by solving the resulting linear system. Since (\ref{42}) is transpose of (\ref{41}), it does not give any additional information. Finally, to compute a matrix ${\tilde{X}}_{n-r,n-r}$ to satisfy (\ref{43}), it is sufficient to choose an arbitrary lower triangular nonsingular matrix $L_{n-r,n-r}$ and substitute it in (\ref{43}). The resulting ${\tilde{X}}_{n-r,n-r}$ gives a symmetric positive definite $\tilde{X}$ as follows:
\begin{center}
$\tilde{X}=\left(
                                         \begin{array}{cc}
                                           {\tilde{X}}_{rr} & {\tilde{X}}_{r,n-r} \\
                                           {\tilde{X}}_{n-r,r} & {\tilde{X}}_{n-r,n-r} \\
                                         \end{array}
                                       \right).$
\end{center}
                                       Now, based on the above discussion, we outline the steps of our algorithm for solving (\ref{6}) in the case $rank(D)=r<n.$\\\\
\newpage

\textbf{Solving the EIV model for positive definite linear system with rank deficient data and target matrices using spectral decomposition.}\\

 $\delta$ as the upper bounds for absolute error is taken to be close to the machine (or user's) zero.\\
 Let $A=D^TD$ and compute its spectral decomposition:
 \begin{center}
$A=U\left(
\begin{array}{cc}
S^2 & 0 \\
0 & 0
\end{array}\right)U^T.$
\end{center}
 Let $B=T^TT$ and $\tilde{B}=U^TBU$.\\
 Compute $rank(D)=r$ and let
\begin{eqnarray}\nonumber
{\tilde{B}}_{rr}=\tilde{B}(1:r,1:r),\nonumber\\
{\tilde{B}}_{r,n-r}=\tilde{B}(1:r,r+1:n),\nonumber\\
{\tilde{B}}_{n-r,n-r}=\tilde{B}(r+1:n,r+1:n)\nonumber
\end{eqnarray}
 Let $\bar{D}=S$, assume $\bar{T}$ satisfies ${\tilde{B}}_{rr}={\bar{T}}^T{\bar{T}}$.\\
 Perform \textbf{Algorithm 1} with input parameters $D=\bar{D}$ and $T=\bar{T}$, and let ${\tilde{X}}_{rr}=X^{\ast}$.\\
 Solve the linear system (\ref{23}) to compute ${\tilde{X}}_{r,n-r}$ and let ${\tilde{X}}_{n-r,r}={{\tilde{X}}_{r,n-r}}^T$.\\
 Compute the spectral decomposition for $B$, that is, $B=V\left(
                                         \begin{array}{cc}
                                           D^2 & 0 \\
                                           0 & 0 \\
                                         \end{array}
                                       \right)V^T.$
 Compute $M={U_r}^TV_r$.\\
If $\|{U_{n-r}}^T(BU_r{({U_r}^T BU_r)}^{-1}{U_r}^T B-B)\|\geq \delta$ \textbf{stop} ((\ref{6}) has no solution)\\
Else\\
Let the Cholesky decomposition of ${\tilde{X}}_{rr}$ be ${\tilde{X}}_{rr}=\tilde{L}{\tilde{L}}^T$ and set $L_{rr}=\tilde{L}$.\\
 Solve the lower triangular system (\ref{41}) to compute $L_{n-r,r}$.\\
 Let $L_{n-r,n-r}\in {\mathbb{R}}^{(n-r)\times (n-r)}$ be an arbitrary nonsingular lower triangular matrix and compute ${\tilde{X}}_{n-r,n-r}$ using (\ref{43}).\\
 Let $\tilde{X}=\left(
                                         \begin{array}{cc}
                                           {\tilde{X}}_{rr} & {\tilde{X}}_{r,n-r} \\
                                           {\tilde{X}}_{n-r,r} & {\tilde{X}}_{n-r,n-r} \\
                                         \end{array}
                                       \right)$ and $X^{\ast}=U\tilde{X}U^T$.\\
 Compute $E=\tr((DX^{\ast}-T)(D-T{X^{\ast}}^{-1})).$\\
EndIf.\\\\
Next, we show how to use the complete orthogonal decomposition of the data matrix $D$ instead of the spectral decomposition of $A$.\\

\textbf{Note} (Complete Orthogonal Decomposition) \cite{27}
Let $A\in {\mathbb{R}}^{m\times n}$ be an arbitrary matrix with $rank(A)=r$. There exist $R \in {\mathbb{R}}^{r\times r}$, $U\in {\mathbb{R}}^{m\times m}$ and $V\in {\mathbb{R}}^{n\times n}$ so that $R\in {\mathbb{R}}^{r\times r}$ is upper triangular, $UU^T=U^TU=I$, $VV^T=V^TV=I$ and $A=U\left(
                                         \begin{array}{cc}
                                           R & 0 \\
                                           0 & 0 \\
                                         \end{array}
                                       \right)V^T.$\\\\
                                       
Next, \textsc{Algorithm 4} is presented using the complete orthogonal decomposition of $D$.\\
\newpage

\textbf{Solving the EIV model for positive definite linear system with rank deficient data and target matrices using complete orthogonal decomposition.}\\

 $\delta$ as the upper bounds for absolute error is taken to be close to the machine (or user's) zero.\\
 Compute the complete orthogonal decomposition of $D$, that is,
\begin{equation}\nonumber
D=U\left(
                                         \begin{array}{cc}
                                           R & 0 \\
                                           0 & 0 \\
                                         \end{array}
                                       \right)V^T.
                                       \end{equation}
 Let $A=D^TD=V_rR^TR{V_r}^T$, $B=T^TT$ and $\tilde{B}=V^TBV$, where $V_r$ consists of the first $r$ columns of $V$.\\
 Compute $rank(D)=r$ and let
\begin{eqnarray}\nonumber
{\tilde{B}}_{rr}=\tilde{B}(1:r,1:r),\nonumber\\
{\tilde{B}}_{r,n-r}=\tilde{B}(1:r,r+1:n),\nonumber\\
{\tilde{B}}_{n-r,n-r}=\tilde{B}(r+1:n,r+1:n).\nonumber
\end{eqnarray}
 Let $\bar{D}=R$, assume $\bar{T}$ satisfies ${\tilde{B}}_{rr}={\bar{T}}^T{\bar{T}}$.\\
 Perform \textbf{Algorithm 1} with input parameters $D=\bar{D}$ and $T=\bar{T}$, and let ${\tilde{X}}_{rr}=X^{\ast}$.\\
 Solve the linear system (\ref{23}) to compute ${\tilde{X}}_{r,n-r}$ and let ${\tilde{X}}_{n-r,r}={{\tilde{X}}_{r,n-r}}^T$.\\
 Compute the spectral decomposition for $B$, that is, $B=V\left(
                                         \begin{array}{cc}
                                           D^2 & 0 \\
                                           0 & 0 \\
                                         \end{array}
                                       \right)V^T.$
 Compute $M={U_r}^TV_r$.\\
If {$\|{U_{n-r}}^T(BU_r{({U_r}^T BU_r)}^{-1}{U_r}^T B-B)\|\geq \delta$}
\textbf{stop} ((\ref{6}) has no solution)\\
Else\\
 Let the Cholesky decomposition of ${\tilde{X}}_{rr}$ be ${\tilde{X}}_{rr}=\tilde{L}{\tilde{L}}^T$ and set $L_{rr}=\tilde{L}$.\\
Solve the lower triangular system (\ref{41}) to compute $L_{n-r,r}$.\\
 Let $L_{n-r,n-r}\in {\mathbb{R}}^{(n-r)\times (n-r)}$ be an arbitrary nonsingular lower triangular matrix and compute ${\tilde{X}}_{n-r,n-r}$ using (\ref{43}).\\
 Let $\tilde{X}=\left(
                                         \begin{array}{cc}
                                           {\tilde{X}}_{rr} & {\tilde{X}}_{r,n-r} \\
                                           {\tilde{X}}_{n-r,r} & {\tilde{X}}_{n-r,n-r} \\
                                         \end{array}
                                       \right)$ and $X^{\ast}=U\tilde{X}U^T$.\\
 Compute $E=\tr((DX^{\ast}-T)(D-T{X^{\ast}}^{-1})).$\\
EndIf.\\\\
Thus, based on the above study, the computational complexity of PDEIV-QR is lower than that of PDEIV-Spec, for all matrix sizes. But, for the case of rank deficient data matrix, depending on the matrix size and rank, one of the algorithms PDEIV-RD-Spec and PDEIV-RD-COD may have a lower computational complexity.\\

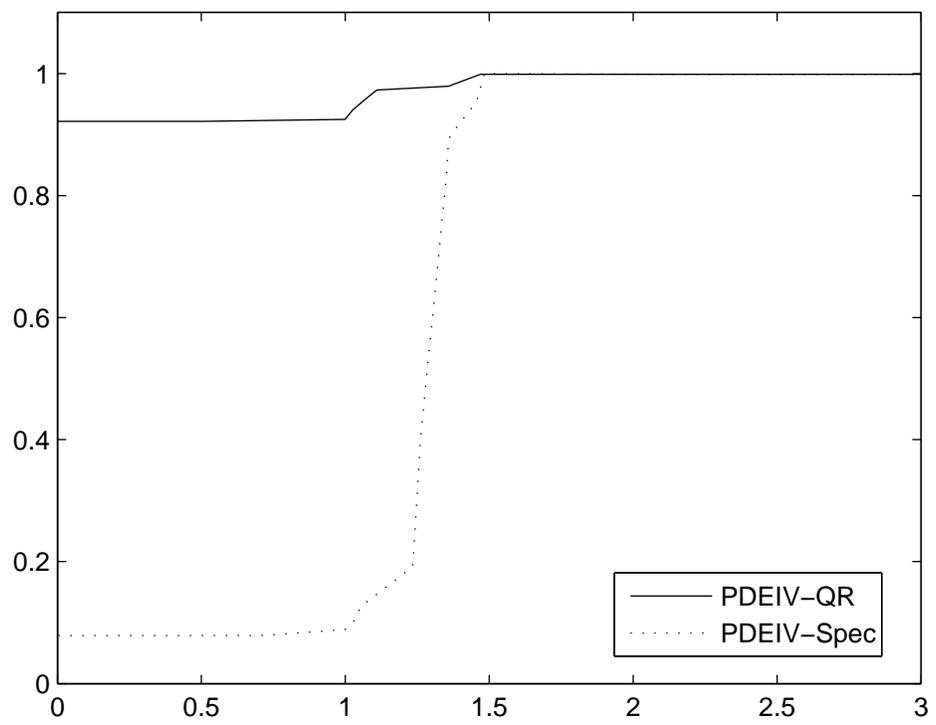
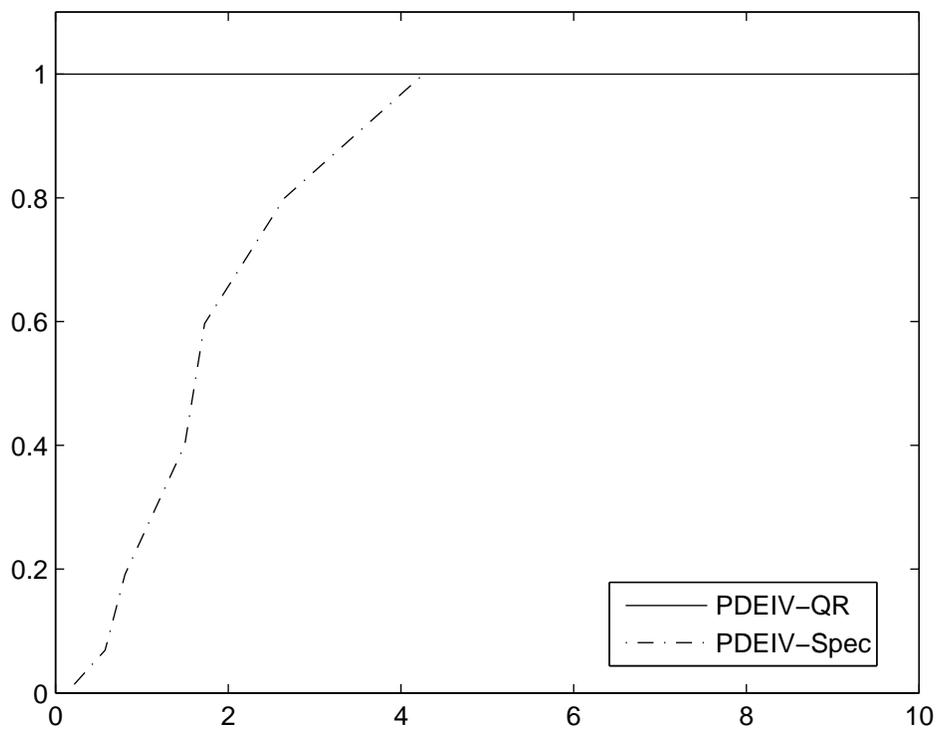
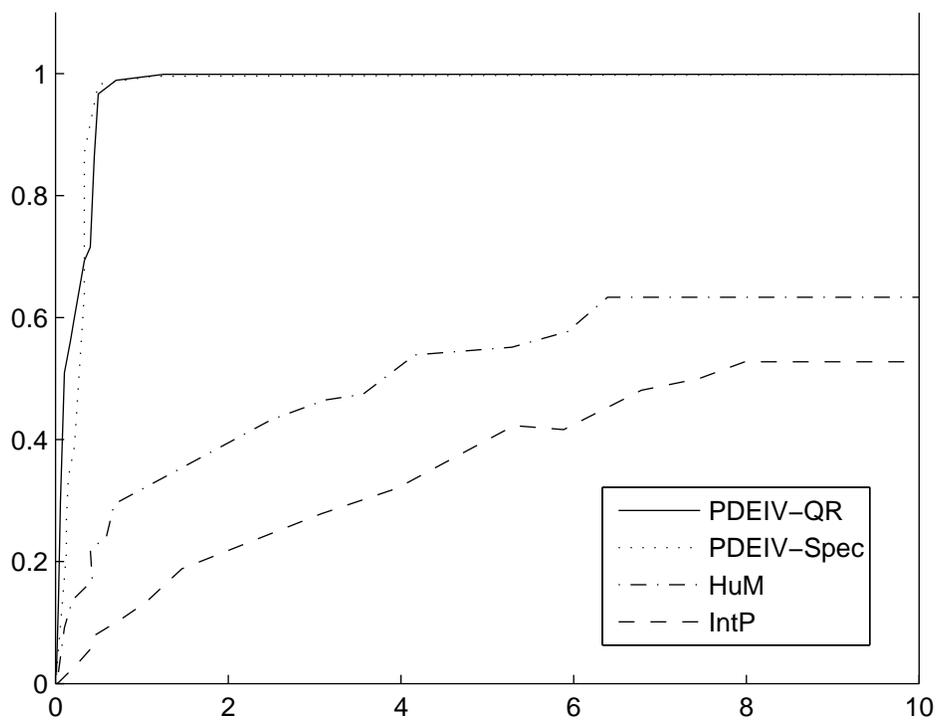
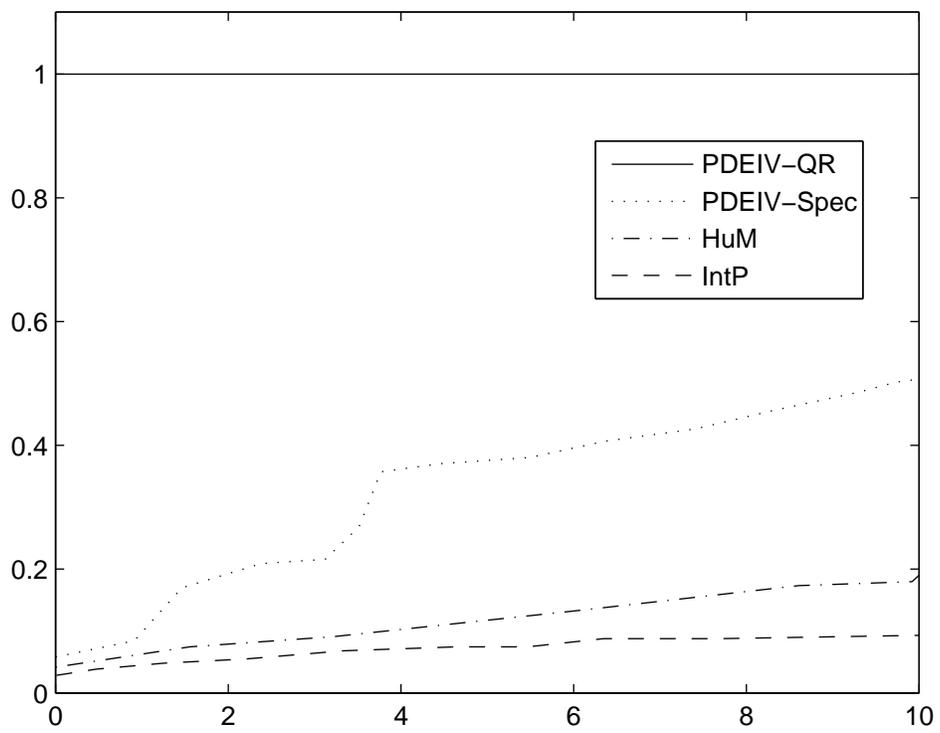
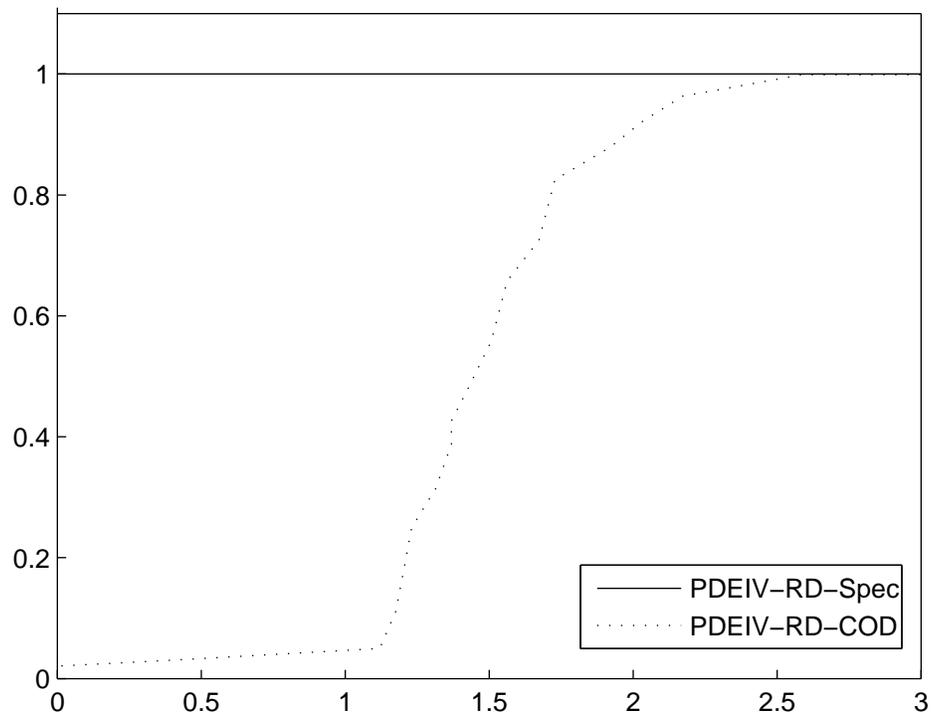
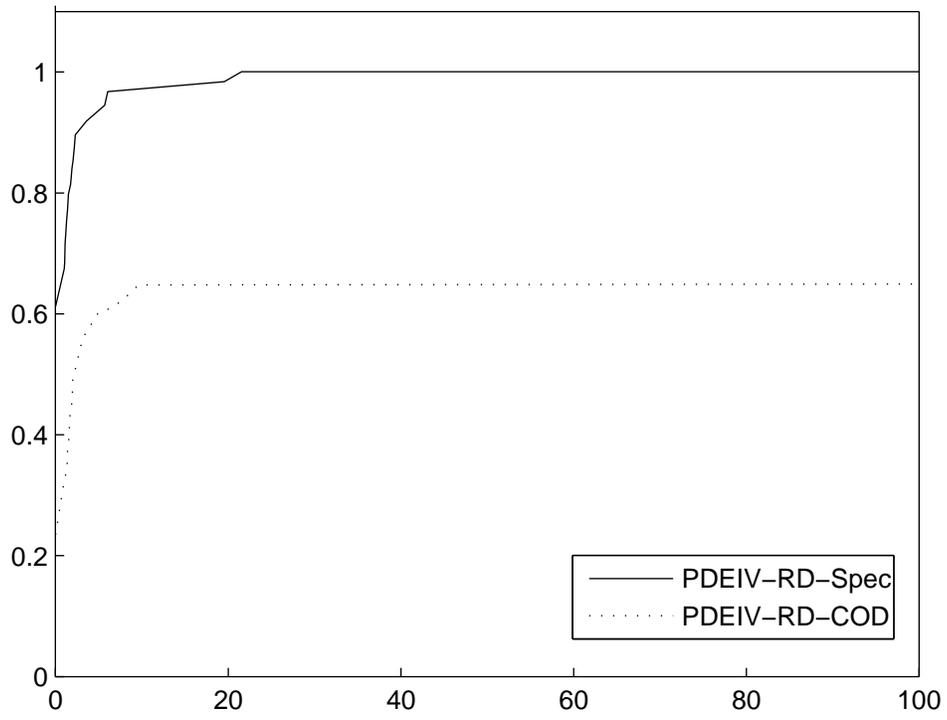
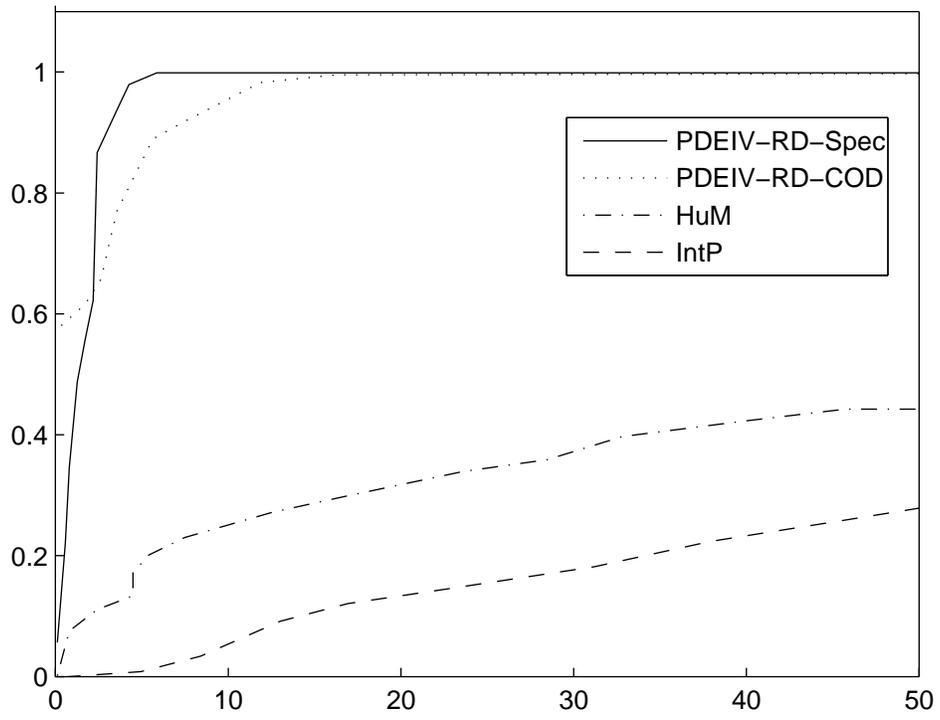
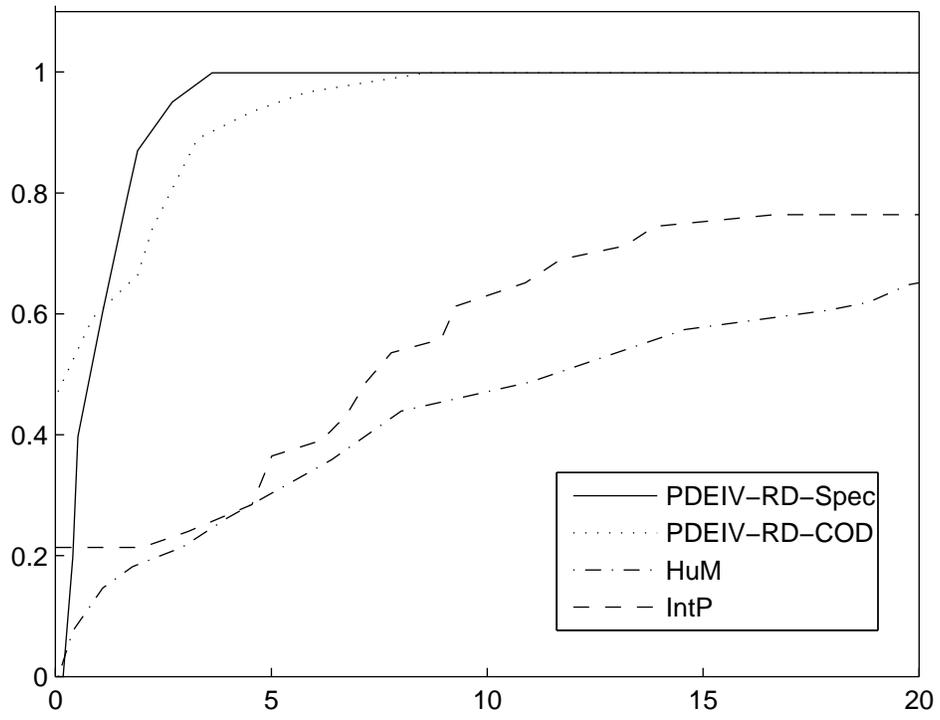
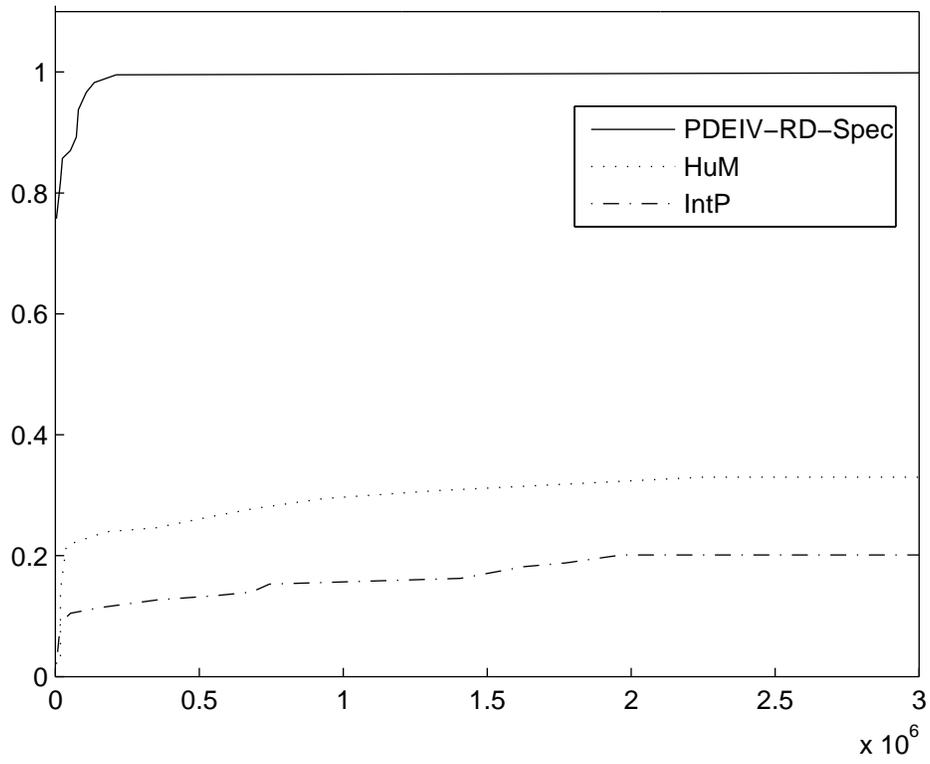
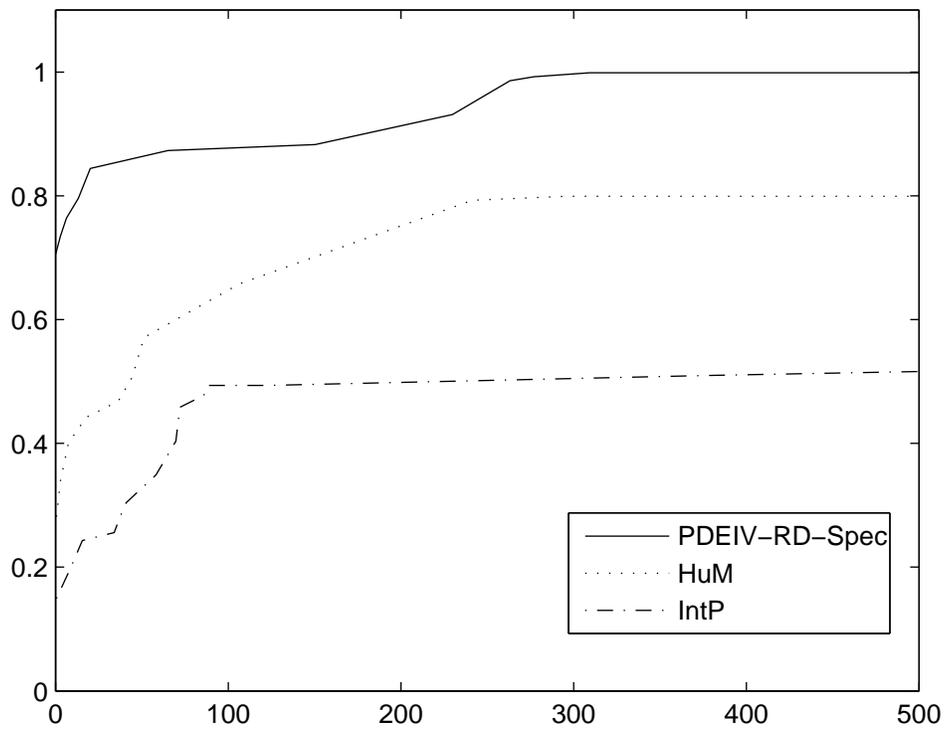
\end{document}